\documentclass[11pt]{amsart}
\usepackage{amssymb,amscd,amsmath}
\usepackage[all]{xy}
\newtheorem{thm}{Theorem}[section]

\newtheorem{cor}[thm]{Corollary}

\newtheorem{prop}[thm]{Proposition}

\newtheorem{lem}[thm]{Lemma}
\theoremstyle{remark}

\theoremstyle{definition}
\newtheorem{defn}[thm]{Definition}

\renewcommand{\bar}{\overline}

\newcommand{\Hdim}{\mathrm{Hdim}}

\newcommand{\A}{{\mathbb{A}}}

\renewcommand{\P}{{\mathbb{P}}}
\newcommand{\F}{{\mathbb{F}}}

\newcommand{\Z}{{\mathbb{Z}}}

\newcommand{\uG}{{\underline{G}}}
\newcommand{\uU}{{\underline{U}}}

\newcommand{\uX}{{\underline{X}}}
\newcommand{\uY}{{\underline{Y}}}
\newcommand{\cG}{{\mathcal{G}}}

\newcommand{\Aut}{\mathrm{Aut}\,}

\newcommand{\Rad}{\mathrm{Rad}\,}
\newcommand{\Soc}{\mathrm{Soc}\,}

\newcommand{\Out}{\mathrm{Out}\,}

\newcommand{\Spec}{\mathrm{Spec}\;}

\newcommand{\s}{\sigma}

\newcommand{\Alt}{{\raise 2pt\hbox{$\scriptstyle\bigwedge$}}}

\newcommand{\go}{\rightarrow}

\newcommand{\e}{\epsilon}

\begin{document}
\title{Words, Hausdorff dimension and randomly free groups}

\author{Michael Larsen}
\email{mjlarsen@indiana.edu}
\address{Department of Mathematics\\
    Indiana University \\
    Bloomington, IN 47405\\
    U.S.A.}

\author{Aner Shalev}
\email{shalev@math.huji.ac.il}
\address{Einstein Institute of Mathematics\\
    Hebrew University \\
    Givat Ram, Jerusalem 91904\\
    Israel}


\thanks{ML was partially supported by NSF grant DMS-1401419.
AS was partially supported by ERC advanced grant 247034, BSF grant 2008194,
ISF grant 1117/13 and the Vinik Chair of mathematics which he holds.}

\begin{abstract}
We study fibers of word maps in finite and residually finite groups, and derive various applications.
Our main result shows that, for any word $1 \ne w \in F_d$ there exists $\e > 0$ such that
if $\Gamma$ is a residually finite group with infinitely many non-isomorphic non-abelian
upper composition factors, then all fibers of the word map $w:\Gamma^d \rightarrow \Gamma$
have Hausdorff dimension at most $d -\e$.

We conclude that profinite groups $G := \hat\Gamma$, $\Gamma$ as above, satisfy no probabilistic identity, and
therefore they are  \emph{randomly free}, namely, for any $d \ge 1$,
the probability that randomly chosen elements $g_1, \ldots , g_d \in G$
freely generate a free subgroup (isomorphic to $F_d$) is $1$.
This solves an open problem from \cite{DPSS}.

Additional applications and related results are also established. For example, combining our
results with recent results of Bors, we conclude that a profinite group in which the set of elements
of finite odd order has positive measure has an open prosolvable subgroup. This may be regarded as
a probabilistic version of the Feit-Thompson theorem.

\end{abstract}

\maketitle

\newpage

\section{Introduction}

In the past few decades there has been considerable interest in the theory of word maps,
see for instance \cite{Bo, LiS, L, S, LS1, LST, GT, N2}, as well as Segal's monograph \cite{Se}
and the references therein.
Many of these works focus on the image of word maps on finite simple groups
and on a related Waring type problem. There is also increasing interest
in fibers of word maps and related problems, see \cite{LS2, LS3, B1, B2}.

In this paper we prove various results showing that fibers of word maps
are small, not just in a measure-theoretic sense, but also in the stronger sense
of Hausdorff dimension. Our results apply for a wide family of finite and infinite groups,
well beyond the family of finite simple groups.

The main result of this paper is the following.

\begin{thm}
\label{newmain}
Let $\Gamma$ be a residually finite group with infinitely many
non-isomorphic non-abelian upper composition factors.
Let $d \ge 1$ and let $w \in F_d$ be a non-trivial word.
Then there exists $\e > 0$ depending only on $w$ such that the
Hausdorff dimension of any fiber of the associated word map
$w:\Gamma^d \rightarrow \Gamma$ is at most $d - \e$.
\end{thm}

Here $F_d$ denotes a free group of rank $d$ freely generated by $x_1, \ldots x_d$.
A word $w = w(x_1, \ldots , x_d) \in F_d$ (which we write in reduced form) gives rise to
a word map $w:\Gamma^d \to \Gamma$ on any group $\Gamma$, which is induced by substitution.

By an \emph{upper composition factor} of $\Gamma$ we mean a composition factor of some
finite quotient $\Gamma/\Delta$ of $\Gamma$, where $\Delta$ is a normal subgroup of $\Gamma$
of finite index (and we assume $\Delta$ is open if $\Gamma$ is a profinite group).

For a residually finite group $\Gamma$ and $d \ge 1$, we define the \emph{Hausdorff dimension}
of a subset $S \subseteq \Gamma^d$ by
\[
\Hdim(S) = \liminf_{\Delta} \frac{\log{|S \Delta^d/\Delta^d|}} {\log{|\Gamma/\Delta|}},
\]
where $\Delta$ ranges over the finite index normal subgroups of $\Gamma$ (and again
if $\Gamma$ is profinite we also assume that $\Delta$ is open).
Thus $\Hdim(\Gamma^d) = d$ and Theorem \ref{newmain} states that,
under the assumptions of the theorem we have
\[
\Hdim(w^{-1}(g)) \le d-\e
\]
for every $g \in \Gamma$.

In fact the proof of Theorem \ref{newmain} gives a bit more; see Theorem \ref{mainfinite}
for an effective finitary version.

These results may be regarded as a far-reaching extension of
the main result of \cite{LS2}, stating that, for $w$ as above there
exists $\e = \e(w) > 0$ such that, if $T$ is a large enough
finite simple group then the fibers of $w:T^d \rightarrow T$ have
size at most $|T|^{d-\e}$. Here and throughout this paper, by a finite simple group
we mean a non-abelian finite simple group.

We now list several consequences of Theorem \ref{newmain}.
The first one deals with linear groups and strengthens results from \cite{LS3}.

It is easy to see, using strong approximation (see \cite{N1, P, W}), that a finitely
generated linear group which is not virtually solvable
has infinitely many finite simple groups of Lie type as upper composition factors.
Applying Theorem~\ref{newmain}, we deduce the following.

\begin{thm}
\label{linear}
Let $\Gamma$ be a finitely generated linear group over any field.
Suppose $\Gamma$ is not virtually solvable. Then the fibers of any word map on
$\Gamma^d$ induced by a non-trivial word $w \in F_d$ have Hausdorff dimension less then $d$.
\end{thm}

The next consequence of our main result deals with probabilistic identities.
We need more notation. Given a word $w \in F_d$ and a finite group $G$, let $p_{w,G}$ denote the associated
probability distribution on $G$. Thus, for $g \in G$ we have $p_{w,G}(g) = |w^{-1}(g)|/|G|^d$.

Recall that a word $1 \ne w \in F_d$ is said to be {\em a probabilistic identity} of a residually finite
group $\Gamma$ if there exists $\delta > 0$ such that, in any finite quotient $H$ of $\Gamma$,
we have $p_{w,H}(1) \ge \delta$.
This amounts to saying that, in the profinite completion $G$ of $\Gamma$, the probability
(with respect to the normalized Haar measure on $G$) that $w(g_1, \ldots, g_d) = 1$ where $g_1, \ldots , g_d \in G$
are random elements is positive.

For example, $w=x_1^2$ is a probabilistic identity of the infinite
dihedral group $\Gamma = D_{\infty}$.

It follows from \cite{Ma} that a residually finite group which satisfies
the probabilistic identity $x_1^2$ is finite-by-abelian-by-finite.
A similar conclusion holds for the probabilistic identity $[x_1,x_2]$,
as follows from the earlier paper \cite{Ne}, which is applied in \cite{Ma}.
However, very little is known about groups satisfying more general probabilistic identities.

Now, let $\Gamma$ be a residually finite group with infinitely many
non-isomorphic non-abelian upper composition factors, and let $1 \ne w \in F_d$.
By Theorem \ref{newmain} we see that
$\Gamma$ has arbitrarily large finite quotients $H$ such that
$p_{w,H}(1) \le |H|^{-\e}$, which tends to $0$.
This implies the following new result.

\begin{thm}
\label{probid}
Let $\Gamma$ be a residually finite group with infinitely many
non-isomorphic non-abelian upper composition factors.
Then $\Gamma$ does not satisfy any probabilistic identity.
\end{thm}

This result generalizes Theorem 1.2 in \cite{LS3}, showing that
a finitely generated linear group which satisfies a probabilistic
identity is virtually solvable.
It also enables us to solve an open problem regarding randomly free groups.

We say that a profinite group $G$ is \emph{randomly free} if, for every $d \ge 1$,
the probability that a $d$-tuple of elements of $G$ freely generates a free subgroup
is $1$. We use \emph{freely generate} in the sense of abstract group theory,
i.e., no non-trivial word evaluated at the chosen elements should give the identity.
A residually finite group is said to be randomly free if its profinite completion
is randomly free.

See \cite{E} and \cite{Sz} for earlier results on groups in
which almost all subgroups are free.

We need the following straightforward observation.

\begin{lem}
\label{rfree}
A residually finite group $\Gamma$ is randomly free if and only if it does not satisfy
any probabilistic identity.
\end{lem}

To show this, let $G$ be the profinite completion of $\Gamma$.
Note that $g_1, \ldots , g_d \in G$ freely generate a free
subgroup of $G$ if and only if $w(g_1, \ldots , g_d) \ne 1$
for every $1 \ne w \in F_d$. Now, suppose $\Gamma$ does not satisfy any probabilistic identity.
Then the probability that $w(g_1, \ldots , g_d) = 1$ ($g_i \in G$) is $0$ for any such word $w$.
As Haar measure is $\sigma$-additive, the probability that there exists $w\neq 1$ such that $w(g_1, \ldots , g_d) = 1$
is also $0$.  Thus, $g_1, \ldots , g_d \in G$ freely generate a free subgroup with
probability $1$, proving that $\Gamma$ is randomly free. The reverse implication is trivial.

Combining the above lemma with Theorem~\ref{probid} we obtain
the following.

\begin{thm}
\label{main}
Let $\Gamma$ be a residually finite group with infinitely many
non-isomorphic non-abelian upper composition factors. Then $\Gamma$ is randomly free.
\end{thm}

This immediately implies (using strong approximation) the following probabilistic Tits alternative,
which is the main result of \cite{LS3}.

\begin{cor} A finitely generated linear group is either virtually solvable
or randomly free.
\end{cor}

Theorem~\ref{main} (applied to profinite groups $G$) solves a problem raised in 2003 in \cite{DPSS}
(see Problem 7 there). A partial solution, assuming $G$ has infinitely
many alternating groups as upper composition factors, was given in 2005 by Ab{\'e}rt,
see \cite[Theorem 1.7]{Ab}.
Hence it remains to prove the theorem assuming $G$ has infinitely many
simple groups of Lie type as upper composition factors.

A recent work of Bors \cite{B1} can be used to handle the case where $G$ has classical
groups of unbounded rank as composition factors. Indeed Theorem 1.1.2 there implies
Theorem \ref{probid} in this case. Hence it remains to handle the case where $G$
has infinitely many groups of Lie type of bounded rank as upper composition factors.

After discussing various consequences of Theorem \ref{newmain}, let us now
state some results of independent interest on finite groups, which play a
crucial role in its proof. The first result deals with almost simple groups
of Lie type of bounded rank.

\begin{thm}
\label{main-bdd}
For any non-trivial word $w = w(x_1,\ldots,x_d)$ and any positive integer $r$, there exist $N, \epsilon > 0$
depending only on $w$ and $r$
such that, if $T$ is a finite simple group of Lie type of order $\ge N$ and rank $\le r$,
then for all $g_1,\ldots,g_d, g \in \Aut(T)$,
$$\bigm|\{(t_1,\ldots,t_d)\in T^d\mid w(t_1 g_1,\ldots, t_d g_d) = g \}\bigm| \le |T|^{d-\epsilon}.$$
\end{thm}

Combining this with results from \cite{LS2} and \cite{B1}, we show that, for any $1 \ne w \in F_d$
there exist $N, \e > 0$ such that for any almost simple group $G$ of size at least $N$
and any $g \in G$ we have $p_{w,G}(g) \le |G|^{-\e}$---see Corollary \ref{almosts} below.
This extends the main result of \cite{LS2} from simple groups to almost simple groups.
This extension is by no means routine, and it occupies a large part of our paper.

Combining the result above with other tools we deduce the following more general
theorem on so-called semisimple finite groups.

\begin{thm}
\label{chief}
Let $G$ be a finite group such that $T^k \le G \le \Aut(T^k)$
for some $k \ge 1$ and a finite simple group $T$.
Suppose $w \ne 1$ is a word. Then there exist constants $N= N(w), \e= \e(w) > 0$
depending only on $w$ such that, if $|T| \ge N$, then for any $g \in G$ we have $p_{w,G}(g) \le |T^k|^{-\e}$.
\end{thm}

Note that, if $G$ is any finite group with a chief factor $G_1/G_2 \cong T^k$, then $G$ has
a semisimple quotient $G/C_G(G_1/G_2)$ (the section centralizer) lying between $T^k$ and its automorphism group, to which
Theorem \ref{chief} may be applied.

The next result enables us to pass from any finite group with a large non-abelian composition factor
to a large semisimple quotient $K$ as above, in which the chief factor $T^k$ is almost as large as
$K$.

\begin{prop}
\label{large}
For any $\delta > 0$ there exists $f = f(\delta) > 0$ such that
if $G$ is a finite group with a non-abelian composition factor of order
$\ge f$, then $G$ has a quotient $K$ with a composition factor $T$ of order
$\ge f$ such that $T^k \le K \le \Aut(T^k)$ and $|T^k| \ge |K|^{1-\delta}$.
\end{prop}

Results \ref{chief} and \ref{large} easily imply our main result, namely Theorem \ref{newmain}.

Finally, combining Theorem \ref{chief} with a new result of Bors \cite[1.1.2]{B2}
we obtain a result which may be regarded as a probabilistic version of the Odd Order Theorem
of Feit and Thompson \cite{FT}. For a finite group $G$ denote by $\Rad(G)$ the solvable radical
of $G$, namely the maximal solvable normal subgroup of $G$.
\begin{thm}\label{odd}
\begin{enumerate}
\item[(i)] Let $k$ be an odd integer and let $w = x^k$. Then for any $\e > 0$ there
is a number $M = M(k,\e)$ such that, if $G$ is any finite group satisfying $p_{w,G}(g) \ge \e$
for some $g \in G$, then
\[
|G/\Rad(G)| \le M.
\]
\item[(ii)] Let $G$ be a profinite group and suppose that the set of elements of $G$ of finite odd order
has positive Haar measure. Then $G$ has a prosolvable open normal subgroup.
\end{enumerate}
\end{thm}

Therefore profinite groups as in part (ii) above are virtually prosolvable.
Part (i) above shows that, for odd $k$, if the probability that $g^k=1$ in $G$ is
bounded away from zero, then $G$ has a solvable normal subgroup of bounded index.

However, it is not true that if the probability
that $g \in G$ has odd order is at least $\e >0$ (or even at least $1-\e$), then $G$
is solvable-by-bounded. Indeed, simple groups of Lie type in characteristic $2$
provide counterexamples (since most of their elements are semisimple, hence of odd order).

A result similar to part (i) of Theorem \ref{odd} with $w = [x_1, \ldots , x_d]$, a left normed
commutator, also follows by combining Theorem 1.1.2 of Bors \cite{B2} with Theorem \ref{chief}
above.

Let us now discuss the strategy of the proof of Theorem~\ref{main-bdd}, which is the
basis of most of our other results. The main idea is, roughly, to convert it to a problem in algebraic geometry.
If, to simplify slightly, $T = \uG(\F_q)$ for some algebraic group $\uG$ of bounded rank,
the parameter $q$ goes to infinity as $|T|\to \infty$.  Any non-trivial word $w$ defines a non-constant morphism $\uG^d\to \uG$.
The fibers are therefore of dimension $\le d\dim \uG-1$, and as $q$ goes to infinity and $g$ varies, the ``complexity'' of the fibers
remains bounded, so we can deduce that
$$|w^{-1}(g)| = O(|\uG(\F_q)|^{d-1/\dim \uG})$$
from standard point counting results for varieties over finite fields.  This idea is not new to this paper (see, e.g., \cite{DPSS, LS2}).
However, there are technical difficulties in implementing it when we must take outer automorphisms of $T$ into account.
For field automorphisms, when $\F_q$ is large but its characteristic $p$ is small, this requires a new idea, namely
finding big gaps between consecutive powers of Frobenius appearing in any specified $d$-tuple of cosets of $T$ (see the proof of Theorem~\ref{main-bdd}, below).
The Suzuki and Ree groups also pose a  technical challenge, since we are no longer counting points on varieties over finite fields but rather taking
fixed points of maps which are square roots of ordinary Frobenius maps.

In \S2, we develop a simple formalism for making precise the idea of bounded complexity mentioned above.
In \S3, we present upper bounds for certain point counting problems.  There are two main variants, one aimed at proving
estimates which are uniform in characteristic and one aimed at dealing with the special difficulties of the Suzuki-Ree case.
In \S4 we prove Theorem \ref{main-bdd} and extend it to all almost simple groups.

Finally, in \S5 we study finite semisimple groups and finite groups in general.
This is where results \ref{chief}, \ref{large}, \ref{newmain} and \ref{odd}
are established.
If $T^k \le G \le \Aut(T^k) \cong \Aut(T) \wr S_k$ as in Theorem \ref{chief}, then $G$
induces a transitive permutation group $P$ on the $k$ copies of $T$, and tools from the theory of permutation
groups  become relevant. Theorem \ref{chief} is proved, roughly, by using the case $k=1$, which has already been established,
and by finding a large set of independent equations induced by the given word equation on $G$.

We then prove Proposition \ref{large} by bounding the order of $P$ above, and deduce
a finitary version---Theorem \ref{mainfinite}---of Theorem 1.1, which readily
implies it.

\bigskip

\section{Degree Bounds}

\begin{defn}
Let $R$ be a commutative ring.   By a \emph{generated commutative algebra} over $R$ (GCA for short),
we mean a pair $(A,S)$ consisting of a commutative $R$-algebra $A$ and a finite set $S$ of generators of $A$ over $R$.
An isomorphism $(A,S)\to (B,T)$ of GCAs is an isomorphism $A\to B$ of $R$-algebras which maps $S$ onto $T$.
\end{defn}

Thus, every GCA is isomorphic to one of the form
$$(R[x_1,\ldots,x_N]/I,\{\bar x_1,\ldots,\bar x_N\}).$$
For $a\in A$ we define $\deg_S a$ to be the minimum integer $m$ such that $a$ can be realized as a degree $m$ polynomial in the elements $S$ with coefficients in $R$.
If $(A,S)$ and $(B,T)$ are GCAs over $R$ and $\phi\colon A\to B$ is an $R$-algebra homomorphism, we define
$$\deg\phi = \max_{s\in S} \deg_T \phi(s).$$
The following lemma is obvious:

\begin{lem}
\label{defs}
With notation as above,
$$\deg_T \phi(a)\le \deg \phi \deg_S a.$$
If $(C,U)$ is also a GCA over $R$, and $\psi\colon B\to C$ is an $R$-algebra homomorphism,
$$\deg \psi\circ \phi\le \deg\phi\deg\psi.$$
If $R'$ is a commutative $R$-algebra and $(A',S')$ denotes the base change of $(A,S)$ to $R'$ (i.e., $A' = A\otimes_R R'$, and $S' = \{s\otimes 1\mid s\in S\}$),
then
$$\deg_{S'} a\otimes 1\le \deg_S a.$$
\end{lem}

The following lemma, asserting that every $q$-Frobenius map has degree at most $q$, is likewise immediate from the definitions:

\begin{lem}
\label{Frob}
If $(A,S)$ is a GCA over $R = \F_q$, and  $\phi\colon A\to A$ denote the $q$-Frobenius map,
then $\deg \phi\le q$.
\end{lem}

We say generating sets $S$ and $T$ of $A$ are \emph{equivalent} if  they generate the same $R$-submodule of $A$.
In this case, $\deg_S a= \deg_{T} a$ for all $a\in A$.  Thus, every surjective $R$-algebra homomorphism $\phi\colon R[x_1,\ldots,x_N]\to A$ determines a well-defined
equivalence class of generating sets in $A$ represented by $\{\phi(x_1),\ldots,\phi(x_N)\}$.  Geometrically, a closed immersion of $\Spec A$ into an affine space over $R$
determines a generating set, and two closed immersions which are the same up to an affine transformation determine equivalent generating sets.

\begin{lem}
\label{Adjoint}
Let $R$ be a field, $\uG=\Spec A$ an adjoint semisimple algebraic group over $R$, $\alpha$ an automorphism of $\uG$ as algebraic group over $R$, and $\rho\colon \uG\to \A^{\dim \uG}$ the adjoint representation.  Up to equivalence, $\rho$ and $\rho\circ\alpha$ determine the same equivalence class of generating sets of $A$.
\end{lem}

\begin{proof}
The homomorphisms $\rho$ and $\rho\circ\alpha$ give the same representation, which means they are conjugate, so they define the same equivalence class of generating sets.
\end{proof}

In general, given an adjoint semisimple algebraic group $\cG$ over a field $\F_q$, we regard its coordinate ring $A$
as a GCA endowed with a generating set $S$ belonging to this equivalence class.

\begin{prop}
\label{deg-bound}
Let $\uG := \Spec A$ be an adjoint semisimple group over $\F_p$ with root system $\Phi$.
Let $F_p$ denote the $p$-Frobenius endomorphism of $\uG$.
Let $\beta_1,\ldots,\beta_l$ be automorphisms of $\uG$ defined over $\F_p$, and let
$n_1, \ldots , n_l \in \{ 1, \ldots , d \}$.
Then, for $\uG$ as above, the morphism $\uG^d\to \uG$ given by
\begin{equation}
\label{gm}
(g_1,\ldots,g_d)\mapsto F_q^{j_1}\beta_1(g_{n_1})^{\pm1}\cdots F_q^{j_l}\beta_l(g_{n_l})^{\pm1}
\end{equation}
has degree $O(p^j)$, where $j=\max_i j_i$, and the coordinate rings of $\uG$ and $\uG^d$ have generating sets $S$ and $S\coprod\cdots\coprod S$ respectively.  (The implicit constant depends on $\Phi$ and $l$ but not on $p$ or $j$.)
\end{prop}

\begin{proof}
As the multiplication and inversion morphisms on $\cG$ have degrees which are bounded independent of $q$,  by Lemma~\ref{defs} and induction on $l$, it suffices to prove the claim for
$g\mapsto F_p^j \beta(g)$.
The proposition follows from Lemma~\ref{defs}, Lemma~\ref{Frob}, and Lemma~\ref{Adjoint}.

\end{proof}

\bigskip

\section{Point Bounds}

It is well known that a degree $s$ hypersurface in $\A^N$ has at most $s q^{N-1}$ points over $\F_q$.
We present several variants of this observation.

\begin{prop}
\label{intersect}
Let $p$ be prime, $q$ a power of $p$, $\uX\subset \A^N$ over $\bar\F_p$ be the closed subscheme defined by the ideal $(f_1,\ldots,f_n)$,
where $f_i\in \bar\F_p[x_1,\ldots,x_N]$ have degrees $d_1,\ldots,d_n$.
If
$$d_1\cdots d_n \le Aq,$$
then
$$|\uX(\bar\F_p)\cap \A^n(\F_q)| \le (2A+1)^N d_1\cdots d_n q^{\dim \uX}.$$
\end{prop}

\begin{proof}
Let
$$\phi_i(x_0,\ldots,x_N) := x_0^{d_i} f_i(x_1/x_0,\ldots,x_N/x_0).$$
The intersection of the hypersurfaces in $\P^N$ defined by the $\phi_i$ defines a closed subscheme $\bar{\uX}$ of $\P^N$,
and the intersection of this subscheme with $\A^N$ is $\uX$.  Therefore, the number of irreducible components of
$\uX$ over $\bar\F_p$ is less than or equal to the number of components in $\bar{\uX}$, which by a suitable version of B\'ezout's theorem (see \cite[12.3.1]{Fulton}), implies that the number of such components is less than or equal to $d_1\cdots d_n$.
This implies the theorem for $\dim \uX=0$.

We proceed by double induction, first on $\dim \uX$ and then on $N$.  The base case for the inner induction is $N=\dim \uX$, and the proposition holds in this case trivially.  For the induction step, we count pairs $(x,H)$ consisting of a point $x\in \uX(\bar\F_p)\cap \A^n(\F_q)$
and $H$ an $\F_q$-rational hyperplane in $\A^N$ containing $x$ (but not necessarily $0$).  The number of such pairs for a given $x\in \uX(\bar\F_p)\cap \A^n(\F_q)$ is $|\P^{N-1}(\F_q)|$, so the total number is
$$(1+q+\cdots+q^{N-1}) |\uX(\bar\F_p)\cap \A^n(\F_q)|.$$

We say that $H$ is \emph{bad} for $\uX$ if it contains an irreducible component of $\uX$ of dimension $\dim \uX$.
The number of hyperplanes containing a given top-dimensional component is at most
$$1+q+\cdots+q^{N-\dim \uX-1},$$
so the number of bad hyperplanes is less than
$$2q^{N-\dim\uX-1} d_1\cdots d_n.$$
By the induction hypothesis for $N$, the number of pairs $(x,H)$ where $H$ is bad is less than
$$  2 q^{N-\dim\uX-1}d_1\cdots d_n (2A+1)^{N-1}d_1\cdots d_n q^{\dim\uX} \le 2A(2A+1)^{N-1} q^N d_1\cdots d_n.$$
By the induction hypothesis for $\dim \uX$, the number of pairs $(x,H)$ where $H$ is good is less than
$$(q+q^2+\cdots+q^N)  (2A+1)^{N-1}d_1\cdots d_n q^{\dim\uX-1}.$$
Thus, the total number of pairs is less than
$$(q+q^2+\cdots+q^N) (2A+1)^N d_1\cdots d_n q^{\dim\uX},$$
so
$$|\uX(\bar\F_p)\cap \A^n(\F_q)| \le (2A+1)^N d_1\cdots d_n q^{\dim\uX},$$
and the proposition holds by induction.
\end{proof}

\begin{prop}
\label{A-S}
Let $(A,S)$ be a GCA over $\Z$.
Let $(A_q,S_q)$ be a GCA over $\F_q$ such that there exists an isomorphism $\iota\colon A\otimes\bar\F_q\to A_q\otimes_{\F_q}\bar\F_q$ with respect to which
$S$ and $S_q$ are equivalent, and let $\uY = \Spec A_q$.   If $f\in A\otimes \bar\F_q$ is not a zero divisor, then
$$|\{y\in \uY(\F_q)\mid f(y) = 0\}| = O((\deg_S f) q^{\dim \uY-1}).$$
\end{prop}

Here the implicit constant depends on $(A,S)$ but not on $q$.

\begin{proof}
Let $S = \{s_1,\ldots,s_M\}$ and $S_q = \{t_1,\ldots,t_N\}$,
and consider the homomorphisms $\phi\colon \Z[x_1,\ldots,x_N]\to A$ and $\phi_q\colon \F_q[y_1,\ldots,y_N]\to A_q$ mapping $x_i$ to $s_i\otimes 1$
and $y_i$ to $t_i\otimes 1$ respectively.
The equivalence of generating sets guarantees the existence of a diagram
$$\xymatrix{\bar\F_q[x_1,\ldots,x_N]\ar[r]\ar[d]_{\phi\otimes 1}&\bar\F_q[y_1,\ldots,y_N]\ar[d]^{\phi_q\otimes 1}\\ A\otimes\bar\F_q\ar[r]^\iota&A_q\otimes_{\F_q}\bar\F_q}$$
where the top arrow maps each $x_i$ to a linear combination of the $y_j$.
If $f_1,\ldots,f_n$ is a generating set for $\ker\phi$.
In particular, $A_q\otimes_{\F_q}\bar \F_q$ is the quotient of $\bar\F_q[y_1,\ldots,y_N]$
by polynomials $\iota(f_1\otimes n),\ldots,\iota(f_k\otimes 1)$ of degrees at most $\deg f_1,\ldots,\deg f_n$
respectively.

The proposition now follows  by applying Proposition~\ref{intersect} to
$$\uX = \Spec \bar\F_p[x_1,\ldots,x_N]/(\iota(f_1\otimes 1),\ldots,\iota(f_n\otimes 1),f).$$
\end{proof}

\begin{prop}
\label{SR-setup}
Let $q=p^f$,  $(A,S)$  a GCA defined over $\F_p$, and
$\uX = \Spec A$.  For all $a\in A$, we let $\uX_a$ denote $\Spec A/(a)$ regarded as a closed subscheme of $\uX$.
Let $k$ be a positive integer prime to $f$, $F_q\colon \uX\to \uX$ the $q$-Frobenius morphism,
and $F\colon \uX\to \uX$ an endomorphism such that $F^k = F_q$.   If $\dim \uX_a < \dim \uX$
and $t$ is a positive integer, then
$$|\uX_a(\bar\F_p)\cap \uX(\bar\F_p)^{F^t}| \le O((\deg_S a) p^{(tf/k) (\dim \uX-1)}),$$
where the implicit constant does not depend on $t$ or $a$.
\end{prop}

\begin{proof}
Let $g$ denote the g.c.d. of $t$ and $k$.  Replacing $F$ by $F^g$ and $t$ and $k$ by $t/g$ and $k/g$ respectively, we may assume without loss of generality that
$t$ and $k$ are relatively prime.

Every endomorphism of a commutative ring maps the nilradical to itself, so without loss of generality, we may assume that
$\uX$ is reduced.  We may also assume that the irreducible components $\uX_1,\ldots,\uX_h$ of $\uX$ are permuted transitively by powers of $F$.

By induction, we may assume that the proposition holds for all affine schemes of dimension less than $\dim \uX$.
Thus, if $\uU$ is any $F$-stable dense open affine subvariety of $\uX$, it suffices to prove the proposition for $\uU$.
If $\uU$ is any dense open affine subvariety of $\uX$, the orbit of $\uU$ under $\langle F\rangle$ is finite, so we may replace $\uX$
by an $F$-stable dense open affine subvariety of $\uU$.  Since every quasi-affine variety contains a dense affine subvariety, it suffices to prove
the desired estimate after replacing $\uX$ by our choice of dense quasi-affine $\uU\subseteq \uX$.

Applying this observation to the complement of
$$\bigcup_{1\le i < j \le h} \uX_i\cap \uX_j$$
we may assume that
the connected components of $\uX$ coincide with the irreducible components $\uX_i$, so the proposition is trivial unless $t = hs$ for some integer $s$.  We may therefore assume that $h$ is relatively prime to $k$. Replacing $F$ by $F^h$ and $t$ and $f$ by $s$ and $fh$ respectively, we may assume without loss of generality that $F$ fixes every component of $\uX$,
so we may reduce to the case that $\uX$ itself is irreducible.

If $K$ denotes the fraction field of $A$, then $F^*$ defines an inclusion $K\hookrightarrow K$, and
 $(F^k)^* = (F_q)^*$.  As $K$ is finitely generated, the latter inclusion is of degree $q^{\dim \uX}=p^{f\dim \uX}$.
By hypothesis, $k$ is relatively prime to $f$, so
$k$ divides $\dim \uX$.

Let $\dim \uX = mk$.  We claim that there exists a  transcendence basis
$b_{11},\ldots,b_{1k}, b_{21},\ldots,b_{mk}$  for $K$ in $A$  such that
$$F^* b_{ij} = \begin{cases}b_{i\,j+1}&\hbox{ if $j< k$,}\\ b_{i1}^{q}&\hbox{ otherwise.}\end{cases}$$
Indeed, by the same reasoning that shows that the transcendence degree of $K$ is a multiple of $k$, we see that every $F^*$-stable subfield of $K$
has transcendence degree a multiple of $k$.
Thus, we can construct the $b_{i1}$ iteratively, defining $b_{q+1\,1}$ to be any element of $A$ not algebraic over the (algebraically independent) set $\{b_{ij}\mid i\le q\}$ and setting
$b_{i\,j+1} = (F^j)^* b_{i1}$.

Endowing $\A^{mk} = \Spec \F_p[x_{ij}]$ ($1\le i\le m$, $1\le j\le k$) with the $F$-action given by
\begin{equation}
\label{standard}
F^* x_{ij} = \begin{cases}x_{i\,j+1}&\hbox{ if $j< k$,}\\ x_{i1}^{q}&\hbox{ otherwise,}\end{cases}
\end{equation}
we see that the map $x_{ij}\mapsto b_{ij}$ is a generically finite $F$-equivariant morphism $\uX_1\to \A^{mk}$.
Passing to a sufficiently small $F$-stable affine open subset $\Spec A[1/h]$
of the base, we may assume that $A[1/h]$ is a finitely generated free $\F_p[x_{ij}][1/h]$-module of some rank $\rho$.
We can therefore realize $A[1/h]$ as a commutative subring of $M_\rho(\F_p[x_{ij}][1/h])$.
In particular, all elements of $S$ can be realized in this matrix ring, and it follows that any polynomial of degree $\deg_S a$ in the elements of $S$ can be realized as a matrix whose
entries are polynomials of degree $O(\deg_S a)$ in the $x_{ij}$ and $1/h$.  The determinant of this matrix is a polynomial of degree
$O(\deg_S a)$ in $x_{ij}$ and $1/h$.  Every point in the zero locus $\uX_a$ maps to a point in the zero locus of this
polynomial.  Since the map $\Spec A[1/h]\to \Spec \F_p[x_{ij}][1/h]$ is of degree $\rho$, it suffices to bound the number of $F^t$-points
on the zero locus of this determinant.  Expressing the determinant as a rational function in the $x_{ij}$, the numerator is a polynomial of degree $O(\deg_S a)$,
and it suffices to bound the number of $F^t$-points of the zero locus of the numerator, viewed as a hypersurface in $\Spec \F_p[x_{ij}]$.  Thus,
we can reduce the general problem to the case
$\uX = \Spec \F_p[x_{ij}]$, with the  $F$-action given by (\ref{standard}).

More explicitly, it suffices to prove that if $Q$ is a polynomial in the $x_{ij}$,
$$|\{(a_{ij})\in \bar\F_p^{mk}\mid Q(a_{ij})=0\}^{F^t}| = O(\deg Q (t/k)p^{mk-1}).$$
We project $\A^{mk}$ to $\A^k$ by
$$\pi_m\colon (a_{ij})\mapsto (a_{m1},\ldots,a_{mk}).$$
Thus $\pi_m$ is $F$-equivariant, where
$$F(c_1,\ldots,c_k) = (c_2,c_3,\ldots,c_k, c_1^q),$$
so
\begin{align*}
F^t(c_1,\ldots,c_k)
&= \bigl(c_{t+1}^{q^{\lfloor t/k\rfloor}}, c_{t+2}^{q^{\lfloor (t+1)/k\rfloor}}, \ldots, c_{t+k-1}^{q^{\lfloor (t+k-1)/k\rfloor}}\bigr) \\
& = \bigl(c_{t+1}^{q^{\lfloor t/k\rfloor}},\ldots,c_k^{q^{\lfloor t/k\rfloor}},c_1^{q^{\lfloor t/k\rfloor+1}},\ldots,c_t^{q^{\lfloor t/k\rfloor+1}}\bigr) \end{align*}
where the $c_i$ are numbered cyclically.
Since $s$ and $k$ are relatively prime, any $F^t$-fixed $(c_1,\ldots,c_k)$ is determined by $c_1$ satisfying $c_1^{p^t}=c_1$.
The number of possibilities for $c_1\in \bar\F_q$ is therefore $p^t$.
Any $F^t$-fixed point $(a_{ij})$ projects under $\pi_m$ to a $F^t$-fixed point.

There are at most $\deg Q$ factors of $Q$ over $\bar\F_p[x_{ij}]$ of the form $x_{m1}-c_1$, and for any $c_1\in \bar\F_p$ such that $x_{m1}-c_1$
is not such a factor and $(c_1,\ldots,c_k)$ is fixed by $F^t$, the fiber $\pi_m^{-1}(c_1,\ldots,c_k)$ is a hypersurface in $\A^{(m-1)k}$ of degree $\le \deg Q$.
If $(c_1,\ldots,c_m)$ is fixed by $F^t$ and $x_{m1}-c_1$ divides $Q$, then the number of $F^t$ fixed points in $\pi_m^{-1}(c_1,\ldots,c_m)$ is the number of $F^t$
fixed points of $\A^{(m-1)k}$, i.e. $q^{t(m-1)}$.  By induction on $m$, it suffices to consider the base case, $m=1$.

For $m=1$, if
$(c_1,\ldots,c_m)$ is a fixed point of $F^t$, we have
$$c_1 =  c_{t+1}^{q^{\lfloor t/k\rfloor}} = \cdots = c_{(k-1)t+1}^{q^{t-\lceil t/k\rceil}},$$
so $Q(c_1,\ldots,c_k)$ can be expressed as a polynomial in $c_{(k-1)t+1}$ of degree at most
\begin{align*}
(\dim Q)q^{t-\lceil t/k\rceil} &= (\dim Q)p^{tf-\lceil t/k\rceil f} \\
                                         &\le (\dim Q)p^{tf- tf/k} =  (\dim Q)p^{(tf/k)(\dim \uX-1)},
\end{align*}
which proves the proposition.

\end{proof}

\bigskip

\section{Almost simple groups}

In this section we bound the size of fibers of word maps for almost simple groups.
Most of the section is devoted to the proof of Theorem \ref{main-bdd} on almost
simple groups of Lie type of bounded rank.

We make essential use of the following result of Nikolov \cite[Corollary 8]{N2},
ruling out a given coset identity in large almost simple groups.

\begin{prop}
\label{nik}
For every word $1 \ne w \in F_d$ there exists $c_0 = c_0(w)$ such that
if $T$ is a finite simple group of order $\ge c_0$ and $G$ is an almost simple group
with socle $T$, then for every $g_1, \ldots , g_d \in G$ we have
$w(Tg_1, \ldots ,Tg_d) \ne \{ 1 \}$.
\end{prop}

This easily yields the following.

\begin{cor}\label{nik2}
With notation as above, there exists $c_1 = c_1(w)$ such that if $|T| \ge c_1$
then $|w(Tg_1, \ldots , Tg_d)| > 1$.
\end{cor}

\begin{proof} Define $v(x_1, \ldots , x_{2d}) = w(x_1, \ldots , x_d)w(x_{d+1}, \ldots , x_{2d})^{-1}$.

Setting $c_1(w) = c_0(v)$ the result follows.
\end{proof}

Now let $w = x_{n_1}^{e_1} \cdots x_{n_l}^{e_l}$ be a reduced word
of length $l \ge 1$ in $F_d$. This means that $n_1,\ldots, n_l\in \{1,2\ldots,d\}$,
$e_i = \pm 1$, and $n_i = n_{i+1}$ implies $e_i = e_{i+1}$.

Let $G$ be a group, and let $\alpha_1, \ldots , \alpha_l \in \Aut(G)$.
The map $G^d \rightarrow G$ given by
$$(g_1, \ldots , g_d) \mapsto \alpha_1(g_{n_1})^{e_1} \cdots \alpha_l(g_{n_l})^{e_l}$$
is called a \emph{generalized word function} (allowing $w$ to be non-reduced) or an \emph{automorphic word map} on $G^d$ based on $w$
(see \cite[\S 1.3]{Se} and \cite[1.2.1]{B1}).

Let $T \lhd G$ and consider the word map $w:G^d \to G$. Its restriction to $Tg_1\times \cdots\times Tg_d$
(where $g_i \in G$) can be regarded as a map $T^d\to G$ whose image lies in a $T$-coset of $G$.
Indeed, more is true, namely:
\begin{lem}
\label{auto}
With the above notation, the map
$$(t_1,\ldots,t_d)\mapsto w(g_1,\ldots,g_d)^{-1} w(t_1g_1,\ldots, t_d g_d)$$
is an automorphic word map $T^d \to T$ based on $w$.
\end{lem}

This well known observation is easily proved by induction on $l$.

%

\begin{lem}
\label{pigeonhole}
Given integers $j_1,\ldots,j_l$, there exist nonnegative integers $j'_1,\ldots,j'_l\le m-\lceil m/l\rceil$ such that
$$j_1-j'_1\equiv\cdots\equiv j_l-j'_l\pmod m.$$
\end{lem}

\begin{proof}
Without loss of generality, we may assume $0\le j_1\le \cdots\le j_l < m$.
Setting $j_{l+i}=m+j_i$ for $1\le i\le k$, we have $j_{i+1}-j_i\ge 0$ for $i=1,2,\ldots,m$, so $j_{r+1}-j_r \ge m/l$ for some $1\le r\le m$, so setting
$$j'_i=\begin{cases} m-j_{r+1}+j_i&\text{if }1\le i\le r \\ j_i-j_{r+1}&\text{if } r+1\le i\le m,\end{cases}$$
the lemma follows.
\end{proof}

We now prove Theorem \ref{main-bdd}.

\begin{proof}

For every group $T$ as above, there exists a prime $p$,
an adjoint split simple algebraic group  $\uG_0$ over $\F_p$ with root system $\Phi$, and  a generalized Frobenius endomorphism $F$ of $\uG_0$
such that $T$ is the derived group of $\uG_0(\bar{\F_p})^{F}$.
Let $S$ denote a fixed set of generators of the coordinate ring $A$ of $\cG_\Phi$.
Let $F_p$ denote the $p$-Frobenius map of $\uG_0$.
The condition that $F$ is a generalized Frobenius endomorphism means that some positive power of $F$ is a positive power of $F_p$.

Fix cosets of $Tg_1, \ldots , Tg_d$ of $T$ in $\Aut(T)$ and let $w \in F_d$ be as above.
By Lemma \ref{auto} there exist $\alpha_1, \ldots , \alpha_l \in \Aut(T)$ such that
for $(t_1, \ldots , t_d) \in T^d$ the map
\begin{equation}
\label{tm}
(t_1,\ldots,t_d) \mapsto w(g_1,\ldots,g_d)^{-1} w(t_1g_1,\ldots,t_d g_d) = \alpha_1(t_{n_1})^{e_1} \cdots \alpha_l(t_{n_l})^{e_l}
\end{equation}
is an automorphic word map $T^d \to T$ based on $w$.

Each $\alpha_i$ can be written $F_p^{j_i} \beta_i$, where $\beta_i$ is a product of a diagonal automorphism and a graph automorphism.
By Corollary \ref{nik2}, the map (\ref{tm}) is not constant, assuming $T$ is
large enough given $w$.  We can define (\ref{tm}) as the restriction to $T\subset \uG_0(\bar\F_p)$ of a morphism $\uG_0^d\to \uG_0$ defined over $\bar\F_p$
in terms of the comultiplication map on the coordinate ring $A\otimes\F_p$ of $\uG_0$, the inverse map on $A\otimes\F_p$,
and the various endomorphism maps on $A\otimes\bar\F_p$ giving rise to the $\alpha_i$.
This morphism cannot be constant, and by Proposition~\ref{deg-bound}, its degree (with respect to $S_p\coprod\cdots\coprod S_p$) is $O(p^j)$, where $j=\max_{1\le i <d} j_i$.
There exists $s\in S_p$  such that the composition of $s$ and $\uG_0^d\to \uG_0$ gives a non-constant
morphism
\begin{equation}
\label{key-map}
\uG_0^d\to \A^1
\end{equation}
defined over $\bar\F_p$, and the degree of this morphism is $O(p^j)$.

At this point, we divide the problem into the Chevalley-Steinberg case and the Suzuki-Ree case.
In the former,
we can identify $\uG_0(\bar{\F_p})^F$ with $\uG(\F_q)$, where $\uG$ is an adjoint  simple algebraic group over $\F_q$, $q=p^m$,
which becomes isomorphic to $\uG_0$ after extension of scalars from
$\F_q$ to $\F_{q^k}$ and $\F_p$ to $\F_{q^k}$ respectively.  As $F_p^m$ acts as a (possibly trivial) graph automorphism of $T$, we may assume
$0\le j_i < m$ for all $i$.  Applying a suitable power of $F_p$ to $\alpha_1(t_{n_1})^{e_1}\cdots \alpha_l(t_{n_l})^{e_l}$ and using Lemma~\ref{pigeonhole},
we may assume without loss of generality that $0 \le j_i \le m-\lceil m/l\rceil$ for all $i$.  Thus, the $S_p\coprod\cdots\coprod S_p$-degree of the element of
$A^{\otimes d}\otimes \bar\F_p$ defining (\ref{key-map}) is $O(p^{m-\lceil m/l\rceil})$, where by Proposition~\ref{deg-bound}, the implicit constant does not depend on $p$.

We define a generating set for the coordinate ring of $\uG$ by choosing a basis of the adjoint representation
of $\uG$ and letting the generators correspond to matrix entries in the adjoint representation with respect to this basis.  Extending to $\bar\F_p$, this generating set may not be the same as
that obtained by extending $S_p$, but it is equivalent.  By Proposition~\ref{A-S}, the theorem holds for Chevalley and Steinberg groups.

For the Suzuki and Ree cases, we can again use Lemma~\ref{pigeonhole} to assume that the non-constant morphism (\ref{key-map}) has degree $O(p^{m-\lceil m/l\rceil})$.
Applying Proposition~\ref{SR-setup} to this map, we deduce the theorem in these cases.

\end{proof}

We now deduce the following more general result.

\begin{thm}
\label{almostsimp}
For any non-trivial word $w = w(x_1,\ldots,x_d)$ there exist $N, \epsilon > 0$
depending only on $w$ such that, if $T$ is any finite simple group of order $\ge N$,
then for all $g_1,\ldots,g_d, g \in \Aut(T)$,
$$\bigm|\{(t_1,\ldots,t_d)\in T^d\mid w(t_1 g_1,\ldots, t_d g_d) = g \}\bigm| \le |T|^{d-\epsilon}.$$
\end{thm}

\begin{proof}

If $T$ is an alternating group of degree larger than some function $f(w)$ of $w$, or a classical group
of rank larger than $f(w)$, then the conclusion follows from Theorem 3.1.2 in \cite{B1}.
Indeed, the latter result shows that, under our assumption on $T$, there is $\e = \e(w) > 0$ such that
all fibers of all automorphic word maps on $T^d$ based on $w$ have size $\le |T|^{d - \e}$,
which implies the required conclusion.

Otherwise $T$ is a simple group of Lie type of rank $\le f(w)$, and the result follows
by Theorem \ref{main-bdd}.

\end{proof}

Theorem \ref{almostsimp} strengthens of Proposition \ref{nik} of Nikolov: not only
there is no fixed coset identity $w$ in large almost simple groups, the probability
of $w$ attaining any fixed value $g$ on each subset $Tg_1 \times \ldots \times Tg_d$
tends to zero very fast as $|T| \go \infty$.

We conclude with the following probabilistic consequence.

\begin{cor}\label{almosts} For any non-trivial word $w$ there exist $N, \e > 0$ depending only on $w$
such that, for any almost simple group of order at least $N$ and any
element $g \in G$ we have
\[
p_{w,G}(g) \le |G|^{-\e}.
\]
\end{cor}

\begin{proof}
Since $T$ is generated by two elements we trivially have $|G| \le |\Aut(T)| \le |T|^2$
(much better upper bounds hold, of course).
Therefore it suffices to show that, for some $N, \e > 0$ depending on $w$, if $|T| \ge N^{1/2}$
then $p_{w,G}(g) \le |T|^{-2\e}$.

This follows from Theorem \ref{almostsimp} above, by summing up the probabilities over
all subsets $Tg_1 \times \ldots \times Tg_d$ of $G^d$.
\end{proof}

From this, we can immediately deduce Theorem~\ref{chief} in the case of almost simple groups.

\bigskip

\section{Semisimple groups and proof of main theorem}

In this section we prove Theorem \ref{chief} and Proposition \ref{large} and then use them
to deduce Theorems \ref{newmain} and \ref{odd}.

As in \S4, we let $w = x_{n_1}^{e_1} \cdots x_{n_l}^{e_l}$ be a reduced word
of length $l$ in $F_d$.

\noindent
\begin{proof}[Proof of Theorem \ref{chief}]

It is well known that $\Aut(T^k) = \Aut(T) \wr S_k$.
We fix a $d$-tuple of cosets of $T^k$ in $G$: $T^kg_1, \ldots , T^kg_d$
and $g_i = (h_{i1}, \ldots , h_{ik}).\s_i$
where $h_{ij} \in \Aut(T)$ and $\s_i \in S_k$.

Let $w = x_{n_1}^{e_1} \ldots x_{n_l}^{e_l}$, where
$1 \le n_i \le d$, $e_i = \pm 1$.  Note that $n_i = n_{i+1}$ implies $e_i = e_{i+1}$,
since $w$ is reduced.  For $i = 1, \ldots , d$ let $s_i = (t_{i1}, \ldots , t_{ik}) \in T^k$.

Fix $g \in G$. If $w(s_1g_1, \ldots , s_dg_d) = g$ then
\[
(s_{n_1}g_{n_1})^{e_1} \cdots (s_{n_l}g_{n_l})^{e_l} = g,
\]
so
\[
((t_{n_11}h_{n_11}, \ldots , t_{n_1k}h_{n_1k})\s_{n_1})^{e_1}
\cdots ((t_{n_l1}h_{n_l1}, \ldots , t_{n_lk}h_{n_lk})\s_{n_l})^{e_l} = g.
\]
Moving the $\sigma_{n_i}$ terms all the way to the right of the
LHS we can express this equation in the form
\begin{equation}
\label{1}
\tau_1((t_{n_11}h_{n_11}, \ldots , t_{n_1k}h_{n_1k}))^{e_1}
\cdots \tau_l((t_{n_l1}h_{n_l1}, \ldots , t_{n_lk}h_{n_lk}))^{e_l}\sigma = g,
\end{equation}
where $\tau_i, \sigma \in S_k$ do not depend on the $t_{ij}$.
Note that $\tau_i((t_{n_i1}h_{n_i1}, \ldots , t_{n_ik}h_{n_ik}))$
is simply
\[
(t_{n_i\tau_i(1)}h_{n_i\tau_i(1)}, \ldots , t_{n_i\tau_i(k)}h_{n_i\tau_i(k)}).
\]

Let $H \lhd G$ be the kernel of the permutation action of $G$
on the $k$ factors of $T$. Then $G/H \le S_k$ and $T^k \le H \le \Aut(T)^k$.

The equation (\ref{1}) above has no solutions unless $\sigma$ is the image of $g$ in $G/H$.
In the latter case it can be written as a system of $k$ equations. We claim that
each of these is a reduced equation of length $l$ applied to
a $d$-tuple of cosets of $T$. Indeed, if
$(n_i,\tau_i(j)) = (n_{i+1},\tau_{i+1}(j))$ then $n_i = n_{i+1}$
so $e_i = e_{i+1}$. Taking all of these equations together,
each variable $t_{ij}$ appears at most $l$ times.

We obtain $k$ subsets of $\{ t_{ij} \}$, each
of size at most $l$, such that each variable $t_{ij}$ occurs
at most $l$ times. Therefore there exists at least
$m:= \left\lceil \frac{k}{l^2-l+1} \right\rceil$ pairwise disjoint subsets,
corresponding to equations in disjoint sets of variables, which are
clearly independent.

We now apply Corollary \ref{almosts} which implies the case $k=1$ of Theorem \ref{chief}.
We conclude that for some $N, \delta > 0$
depending only on $w$ such that, if $|T| \ge N$, then any one of the $k$ equations discussed above
holds with probability $\le |T|^{-\delta}$. Since our system of equations contains
at least $m$ independent equations, we obtain
\[
p_{w,G}(g) \le (|T|^{-\delta})^m \le |T^k|^{-\delta/(l^2-l+1)}.
\]
Setting $\e = \delta/(l^2-l+1) > 0$ we complete the proof of Theorem \ref{chief}.

\end{proof}

\medskip

The following result, which is Lemma 2.2 of \cite{BCP}, will play a key role
in the proof of Proposition \ref{large} and Theorem \ref{newmain}.

\begin{lem}
\label{BCP}
Fix $c \ge 6$. A permutation group of degree $k$ without composition factors
isomorphic to $A_i$ with $i > c$ has order at most $c^{k-1}$.
\end{lem}

\begin{proof}[Proof of Proposition \ref{large}]

Fix a positive integer $f$. Among all chief factors $G_1/G_2$ of $G$ (where $G_i \lhd G$)
corresponding to non-abelian composition factors of order $\ge f$ choose one of minimal index
(namely $|G:G_1|$ is minimal).

Write $G_1/G_2 = T^k$ for $k \ge 1$ and a finite simple group $T$. Let $C$ be the centralizer
in $G$ of $G_1/G_2$ and let $K = G/C$. Then $T^k \le K \le \Aut(T^k) = \Aut(T) \wr S_k$.
Let $H$ be the kernel of the permutation action of $G$ on the $k$ copies of $T$.
Then $G_1 \le H \lhd G$ and $G/H \le S_k$. By the minimality of $|G:G_1|$ it follows
that all non-abelian composition factors of $G/H$ have order $< f \le |T|$.

Let $c \ge 6$ be minimal such that $G/H$ does not have a composition factor
$A_i$ with $i > c$. Then $|G/H| \le c^{k-1}$ by Lemma \ref{BCP}.
If $A_i$ ($i \ge 5$) is a composition factor of $G/H$ then $2^i < |A_i| < f \le |T|$.
This shows (assuming $f \ge 64$ as we may) that $c \le \log{|T|}$
(where logarithms are to the base $2$).

It is well known that $|\Out(T)| \le \log{|T|}$ for all finite simple groups $T$.
We conclude that
\[
|K| \le |\Aut(T)|^k |G/H| \le |T|^k |\Out(T)|^k (\log{|T|})^k \le |T^k| (\log{|T|})^{2k}.
\]
Now, given $\delta > 0$ choose $f = f(\delta)$ such that $\log{t} \le t^{\delta/2}$
for all $t \ge f$. Since $|T| \ge f$ we obtain
\[
|K| \le |T^k|^{1 + \delta}.
\]
Therefore $|T^k| \ge |K|^{1-\delta}$, completing the proof.

\end{proof}

Combining results \ref{large} and \ref{chief} we obtain the following.

\begin{thm}
\label{mainfinite}
For every word $1 \ne w \in F_d$ there exist constants $N, \e > 0$
depending only on $w$ such that, if $G$ is a finite group with a non-abelian composition
factor $S$ with $|S| \ge N$, then $G$ has a quotient $K$ with $|K| \ge |S|$
such that $p_{w,K}(g) \le |K|^{-\e}$ for all $g \in K$.
\end{thm}

\begin{proof} Let $N(w), \e(w)$ be as in Theorem \ref{chief}.
Fix $\e$ with $0 < \e < \e(w)$.
Define $\delta = 1 - \e/\e(w)$ and let $f = f(\delta)$ be as
in Proposition \ref{large}.  Set $N = \max(f, N(w))$.

Now let $G, S$ be as in the theorem. Since $|S| \ge f$, Proposition
\ref{large} (applied with $|S|$ in the role of $f$) shows that there is
a quotient $K$ of $G$ and a finite
simple group $T$ with $|T| \ge |S|$ such that $T^k \le K \le \Aut(T^k)$
and $|T^k| \ge |K|^{1-\delta}$.

By Theorem \ref{chief} we have
\[
p_{w,K}(g) \le |T^k|^{-\e(w)} \le |K|^{-(1-\delta)\e(w)} = |K|^{-\e}
\]
for every $g \in G$, as required.
\end{proof}

We now prove Theorem \ref{newmain}.

\begin{proof}

This follows immediately from Theorem \ref{mainfinite} applied
to finite quotients of $\Gamma$ with non-abelian composition
factors of orders tending to infinity.

\end{proof}

\noindent
\begin{proof}[Proof of Theorem \ref{odd}]

We first prove part (i).

Theorem 1.1.2 of \cite{B2} shows that, for certain words $w$, including $x^k$ ($k$ odd)
and $[x_1, \ldots , x_d]$, if $G$ is a finite group, $p_{w,G}(g) \ge \e > 0$ for
some $g \in G$, and $T$ is a finite simple group, then the multiplicity
of $T$ as a composition factor of $G$ is bounded above by some function
$f_1(T,w,\e)$ of $T$, $w$ and $\e$ only.

Now, by Theorem \ref{chief} (applied to a suitable quotient of $G$), if $p_{w,G}(g) \ge \e > 0$
and $T$ is a non-abelian composition factor of $G$, then $|T| \le f_2(w, \e)$
for a suitable function $f_2$. This implies that the product of the orders of all non-abelian
composition factors of $G$ is bounded above by
\[
\prod_{T, |T| \le f_2(w, \e)} |T|^{f_1(T, w, \e)} \le f_3(w, \e),
\]
for a suitable function $f_3$.

Now $\Soc(G/\Rad(G))$ has the form $\prod_i T_i^{n_i}$ for (non-abelian) simple
groups $T_i$, hence $|\Soc(G/\Rad(G))| \le f_3(w, \e)$.
Since $G/\Rad(G)$ is embedded in $\Aut(\Soc(G/\Rad(G)))$ we obtain
\[
|G/\Rad(G)| \le M,
\]
where $M = f_4(w, \e) = f_3(w, \e)^{\log f_3(w, \e)}$.
This proves part (i).

For part (ii), note that, since the Haar measure is $\sigma$-additive, there exists
an odd integer $k > 0$ such that the measure of the elements $g$ of the profinite
group $G$ satisfying $g^k= 1$ is positive. The required conclusion now follows from part (i).

\end{proof}

\end{document}